\documentclass[preprint,12pt,authoryear]{elsarticle}

\usepackage{amssymb}
\usepackage{amsmath}

\usepackage{hyperref}      
\usepackage{url}           
\usepackage{booktabs}      
\usepackage{amsfonts}      
\usepackage{nicefrac}       
\usepackage{microtype}     
\usepackage{lipsum}		
\usepackage{graphicx}
\usepackage{natbib}
\usepackage{doi}
\usepackage{subcaption}
\usepackage{amsmath}
\usepackage{amsthm}
\usepackage{float}
\newtheorem{definition}{Definition}
\newtheorem{theorem}{Theorem}

\newtheorem{proposition}{Proposition}

\makeatletter
\newcommand*\owedge{\mathpalette\@owedge\relax}
\newcommand*\@owedge[1]{%
  \mathbin{%
    \ooalign{%
      $#1\m@th\bigcirc$\cr
      \hidewidth$#1\m@th\wedge$\hidewidth\cr
    }%
  }%
}
\makeatother

\usepackage{amsthm}

\journal{Journal of Geometry and Physics}

\begin{document}

\begin{frontmatter}

\title{On the Parallels Between Minimal Surfaces and Einstein Four-Manifolds} 

\author{Mia Beard} 
\affiliation{organization={Mathematical Institute, University of Oxford},%Department and Organization
            addressline={Radcliffe Observatory, Andrew Wiles Building, Woodstock Rd}, 
            city={Oxford},
            postcode={OX2 6GG}, 
            state={Oxfordshire},
            country={United Kingdom}}

%% Abstract
\begin{abstract}
	Minimal surfaces and Einstein manifolds are among the most natural structures in differential geometry:  the former being well understood, the latter far less so. In this exposition, we survey the striking parallels between minimal surfaces in three-manifolds and Einstein metrics in four dimensions. These parallels include variational formulations, topological constraints, monotonicity formulae, compactness and epsilon-regularity theorems, and decompositions such as thick/thin and sheeted/non-sheeted structures. 
    
    Though distinct objects, the striking analogies between them raises a profound question: might there exist circumstances in which these objects are, in essence, manifestations of the same underlying geometry? Drawing on foundational and modern results, this work suggests a bridge between the two structures. In particular, it shows that certain Einstein four-manifolds admit a minimal immersion into a higher-dimensional sphere. As a key example, we realise $ \mathbb{CP}^{2} $ as a minimal submanifold of $ S^{7} $ via the Veronese map, demonstrating a deep unity between the two seemingly distinct worlds.
\end{abstract}

\begin{highlights}
\item Draws a deep structural parallel between minimal surfaces embedded in ambient three-manifolds and Einstein four-manifolds, revealing shared analytical and topological features across distinct geometric settings.
\item Synthesises major results in variational theory, second variation, monotonicity, and compactness to establish a unifying geometric narrative.
\item Introduces dual decomposition frameworks to illuminate localised geometric degeneration in both theories.
\item Proves that under symmetry constraints, a subset of Einstein four-manifolds admit minimal isometric immersions. 
\item Presents a concrete example of the minimal immersion of $ \mathbb{CP}^{2} $ into $ S^{7} $ via the Veronese embedding and Hopf fibration.
\end{highlights}

%% Keywords

\end{frontmatter}

\section{Introduction}
\label{Intro}

What if the soap-film surfaces that minimise area in three dimensions and the elusive Einstein metrics that settle curvature in four dimensions share a hidden bridge? Indeed, at the crossroads of differential geometry, geometric analysis, and global topology lies a series of striking analogies that bind minimal surfaces and Einstein four-manifolds.

In this literature we share explore their shared features, such as the variational formulations, stability under second variation, monotonicity principles, compactness phenomena, and epsilon-regularity theorems, as well as thick/thin and sheeted/non-sheeted decomposition. Motivated by these analogies, we shall synthesise several distinct theorems across geometry to demonstrate that under appropriate conditions, Einstein four-manifolds are guaranteed to admit a minimal immersion into a higher-dimensional ambient sphere - a phenomenon which is equally unexpected and elegant. 

While this article shall explore several established theorems, it does so with the end goal of an original proposition and a key example to back it up.

In this introductory section, we develop the intuition behind each of these structures, so that the reader may fully appreciate the depth and elegance of the similarities.

\subsection{Minimal Surfaces}
\label{MinimalSurfaces}

Minimal surfaces are among the oldest and most fundamental objects in differential geometry. For over 250 years, they have fascinated mathematicians from Euler to Lagrange to Plateau. Defined as surfaces that locally minimise area, they arise as critical points of a variational principle which is satisfied precisely when the mean curvature vanishes across the entire surface. Physically, they manifest as soap films formed when a wire contour is dipped into a soapy solution, as demonstrated by an experiment famously conducted by Plateau in the 19th century. This led to the Plateau problem, which seeks the surface of least area spanning a given closed boundary \citep{Plateau1873}. This problem remains central to the study of minimal surfaces today.

\begin{definition}[Minimal surface]
    Let $ M \subset \mathbb{R}^{3} $ be a smooth surface. We say that $ M $ is minimal if the mean curvature vanishes identically:
    
    \begin{equation}
         H = \frac{1}{2}(k_1 + k_2) = 0 \iff k_{1} = -k_{2},
    \end{equation}

    where $ k_{1} $ and $ k_{2} $ are the principal curvatures of $ M $ at each point.
    
\end{definition}

This geometric condition is one of several equivalent characterisations of minimal surfaces. Analytically, if the surface is expressed as the graph of a function $ u : \Omega \subset \mathbb{R}^2 \to \mathbb{R} $, then the condition of vanishing mean curvature yields the minimal surface equation, which is the Euler-Lagrange equation associated to the area functional:

\begin{equation}
    (1+u_{x}^{2})u_{yy} - 2u_{x}u_{y}u_{xy} + (1 + u_{y}^{2})u_{xx} = 0
\end{equation}

From a variational approach, minimal surfaces arise as critical points of the area functional:

\begin{equation}
    \mathcal{A} = \int_{M} \mathrm{d}A,
\end{equation}

where $ \mathrm{d}A $ denotes the induced area element on the surface $ M $. The variation is typically taken with compact support or with the boundary fixed, depending on the setting.

Classical examples of minimal surfaces include the catenoid, Enneper's surface, Bour's minimal surface, and the gyroid. These surfaces arise not only in mathematical theory, often as solutions to the Plateau problem, but also are important in fields such as architecture, crystal geometry, and material science.

\begin{figure}[H]
  \centering
  \resizebox{0.6\textwidth}{!}{%
    \begin{tabular}{cc}
      \includegraphics[width=0.4\linewidth]{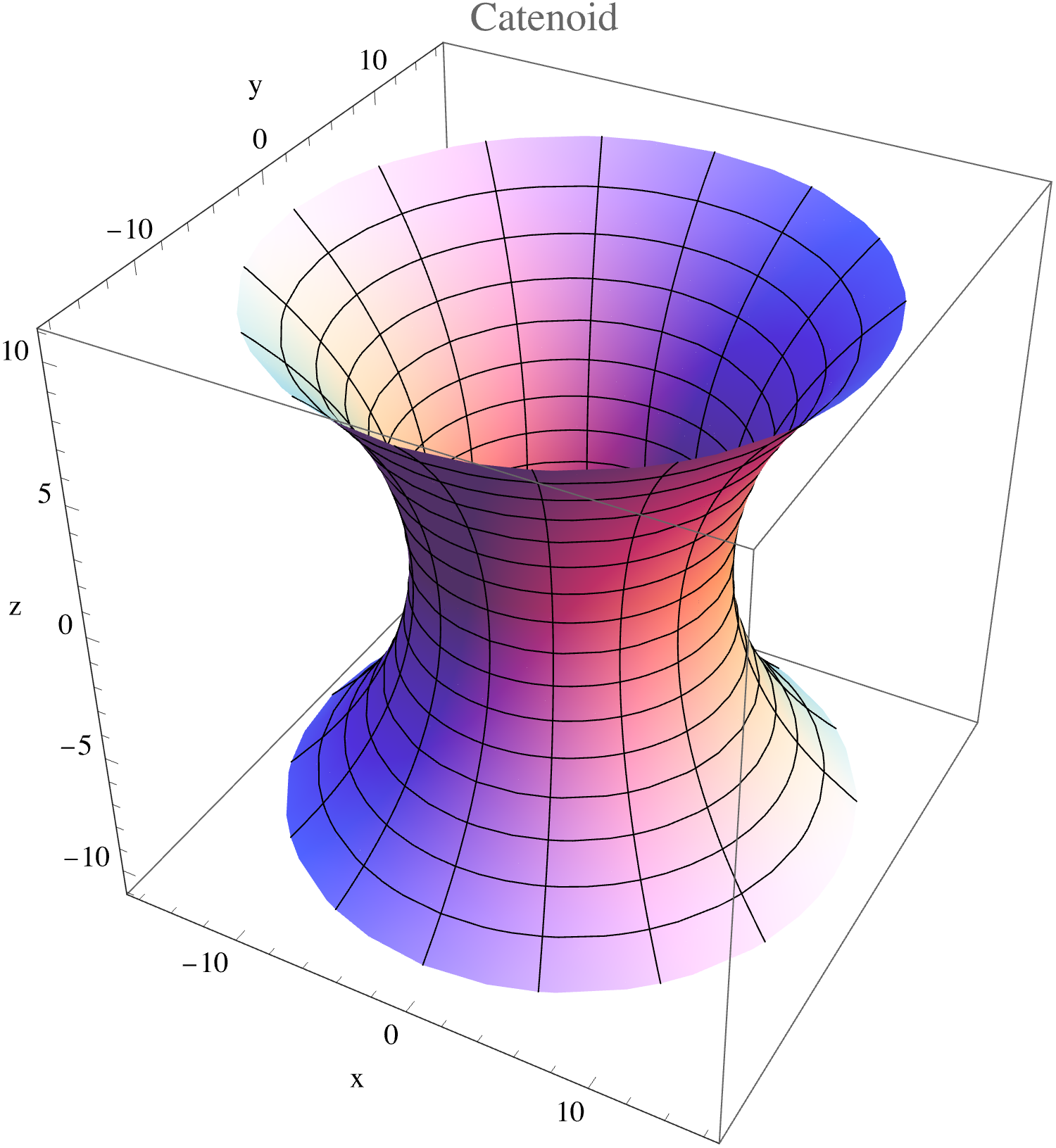} &
      \includegraphics[width=0.4\linewidth]{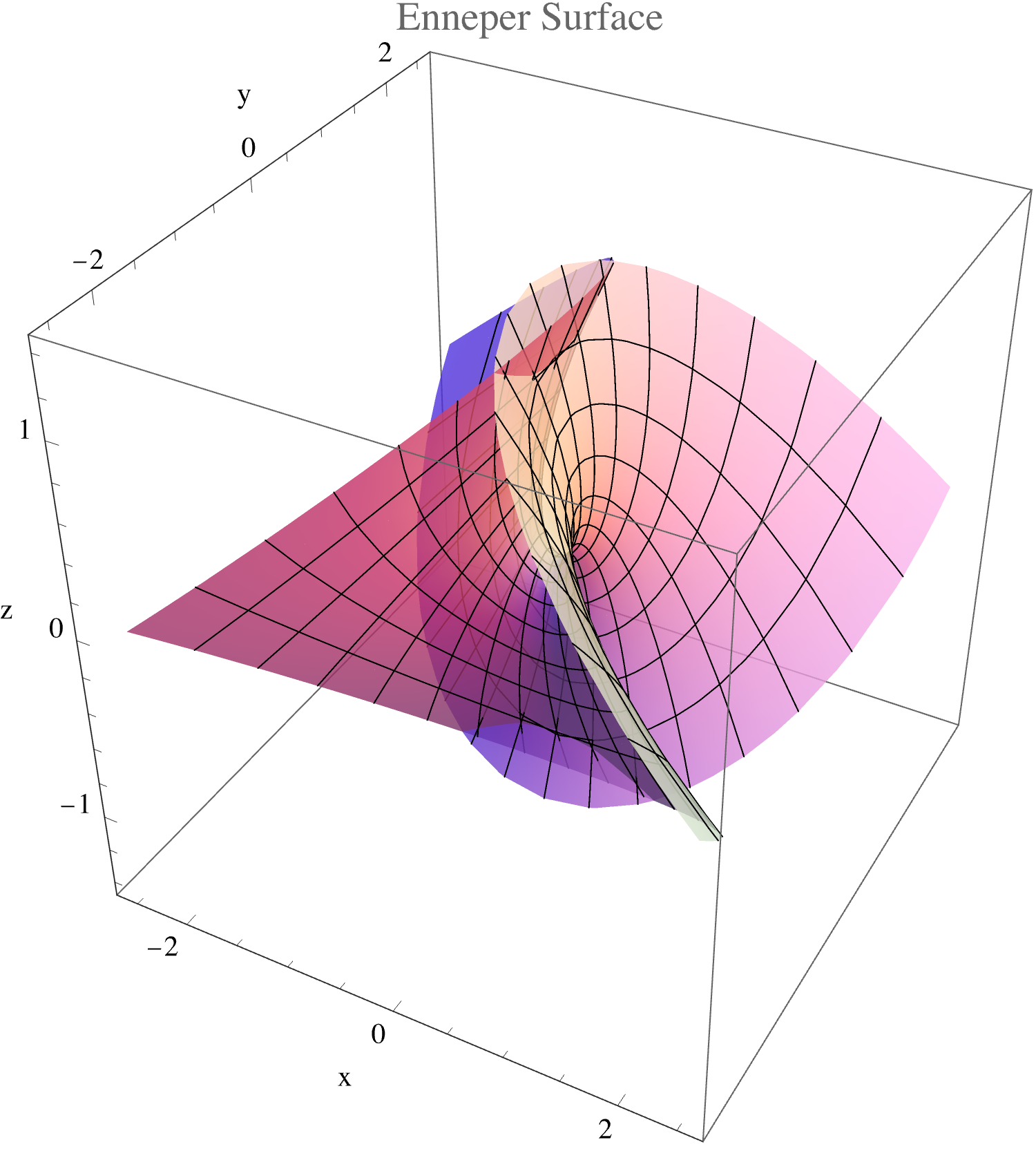} \\[1em]
      \includegraphics[width=0.4\linewidth]{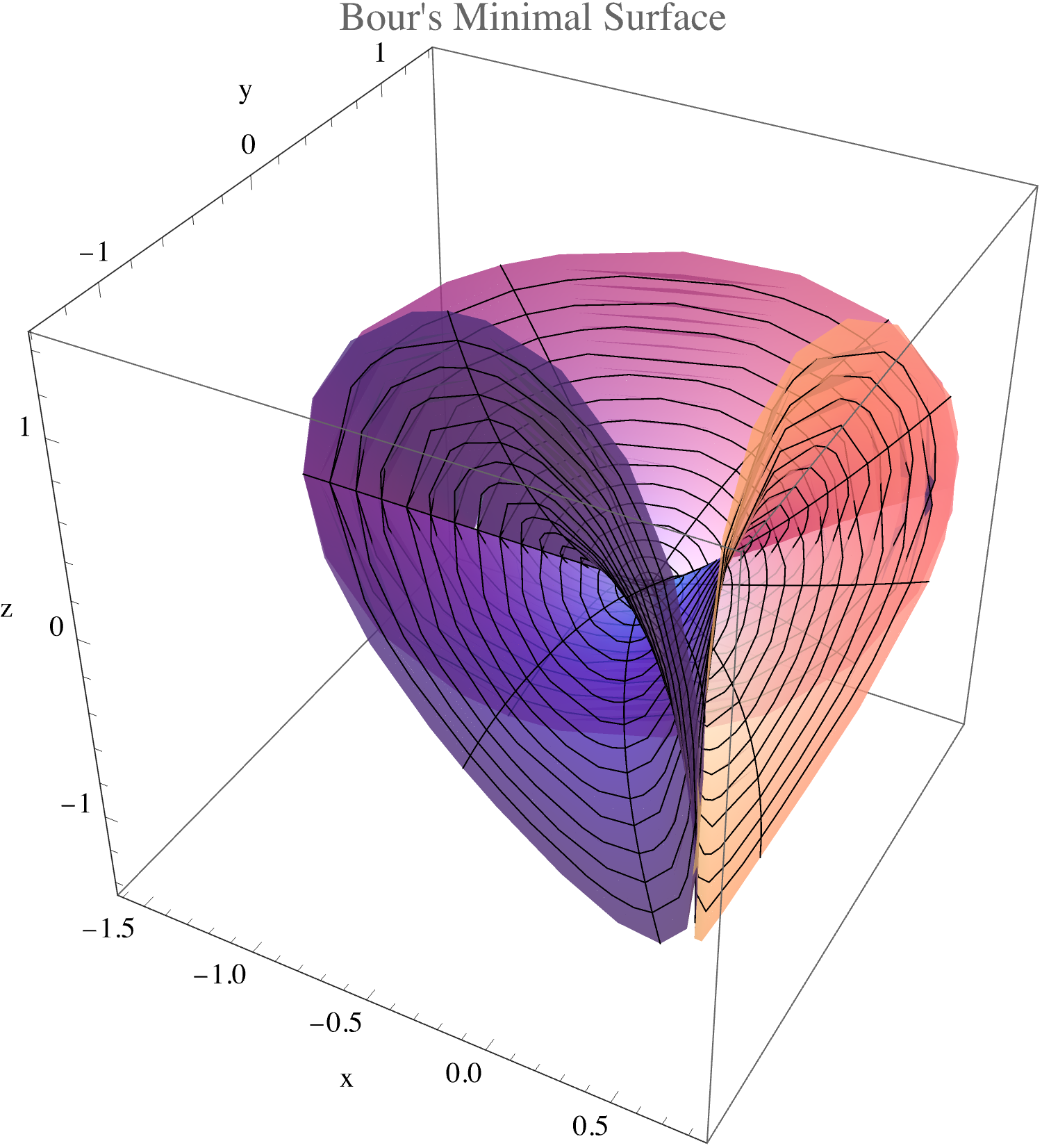} &
      \includegraphics[width=0.5\linewidth]{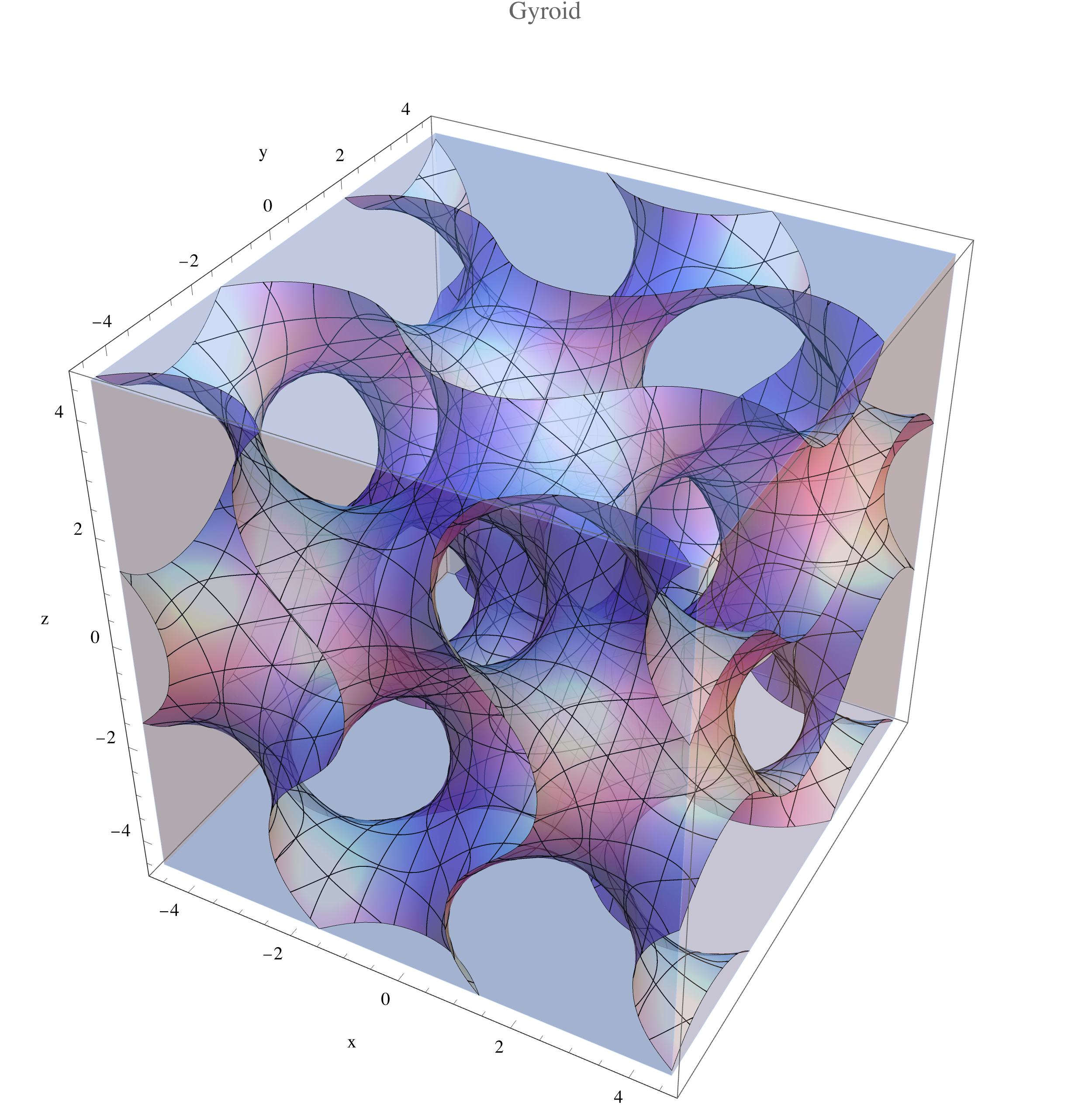}
    \end{tabular}
  }
  \caption{A selection of classical minimal surfaces: the catenoid (top left), the Enneper surface (top right), Bour's minimal surface (bottom left), and the gyroid (bottom right)}
  \label{fig:grid}
\end{figure}

\subsection{Einstein Manifolds}

Einstein manifolds are a distinguished class of  Riemannian manifolds, whose Ricci curvature  is everywhere proportional to the metric tensor, offering a natural generalisation of constant curvature. These manifolds gain physical interpretation in general relativity, where the metric serves as a solution to the vacuum Einstein field equations. We present the following definition, taken from \citep{besse1987einstein}:

\begin{definition}[Einstein manifold]

    A Riemannian manifold $ (M, g) $ is said to be Einstein if there exists a constant $ \lambda \in \mathbb{R} $ such that:

    \begin{equation}
        \operatorname{Ric}_{g} = \lambda g,
    \end{equation}

    where $ \operatorname{Ric}_{g} $ denotes the Ricci curvature tensor of $ g $.
    
\end{definition}

Our knowledge of the structure and classification of Einstein manifolds can be succinctly organised into dimensional cases. 

In dimension two, there is the most elementary manifestation: a Riemannian manifold is Einstein if and only if it has constant Gaussian curvature. These manifolds are locally isometric to the sphere, the Euclidean plane, or the hyperbolic plane.

In dimension three, a Riemannian metric is Einstein if and only if it has constant sectional curvature. This requires the manifold's universal cover to be diffeomorphic to either the three-sphere $ S^{3} $, Euclidean space $ \mathbb{R}^{3} $, or hyperbolic space $ \mathbb{H}^{3} $. This aligns with Thurston's geometrisation conjecture \citep{thurston1980three}, now a theorem \citep{perelman2002entropyformularicciflow, perelman2003finiteextinctiontimesolutions, perelman2003ricciflowsurgerythreemanifolds}, which states that every closed three-manifold decomposes into pieces admitting one of eight canonical geometric structures.

In dimension four, the classification of Einstein manifolds comes significantly more intricate. While some explicit examples are known, much of the theory relies on topological obstructions to the existence of Einstein metrics. Chief among these obstructions is the Hitchin-Thorpe inequality \citep{Thorpe1969, hitchin1974compact}, which states that a compact, oriented four-manifold admitting an Einstein metric must satisfy:

\begin{equation}
    \left| \tau(M) \right| \leq \frac{2}{3} \chi(M),
\end{equation}

where $ \tau(M) $ is the signature and $ \chi(M) $ the Euler characteristic. Equality is achieved in special cases such as K$3$ surfaces with Ricci-flat metrics.

We briefly recall the definitions for these two topological invariants \footnote{The recommendation is that the reader consults \citep{hatcher2002algebraic} for a more thorough understanding, as these definitions rely on the fundamentals of algebraic topology.}:

\begin{definition}[Signature of a Manifold]

Let $ M $ be a closed, oriented, smooth manifold of dimension $ 4k $. The signature $ \tau(M) $ is defined via the intersection form on middle cohomology $ Q_{M} $, where the signature equals the number of positive minus negative eigemvalues of of a matrix representing $ Q_{M} $ \citep{milnor1974characteristic, hirzebruch1995topological, atiyah1968index}.
    
\end{definition}

\begin{definition}[Euler characteristic]
    Let $M$ be a smooth manifold with a finite CW-complex structure. The Euler characteristic  is given by:
    
    \begin{equation}
        \chi(M) = \sum_{k = 0}^{\dim M} (-1)^k c_k,
    \end{equation}

    where $ c_{k} $ denotes the number of $ k $-cells in the complex \citep{hatcher2002algebraic}.
\end{definition}

Examples of Einstein four-manifolds include the 4-sphere $ S^{4} $, the complex projective plane $ \mathbb{CP}^{2} $ (with the Fubini-Study metric \ref{Fubini-Study-Metric}), the K3 surface with a Ricci-flat Kähler metric, and the Eguchi-Hanson space (a non-compact Ricci-flat example).

In higher dimensions, the knowledge of topological restrictions to the existence of Einstein manifold is even more murkey. Many manifolds with dimension greater than four admit a negative Einstein metric \citep{besse1987einstein}. When a positive Einstein metric is required, existence hinges on two more delicate criteria. Myers' Theorem \citep{Myers_1941} states that if a complete Riemannian manifold has Ricci curvature bounded below by a positive constant, then its fundamental group must be finite and its diameter must be bounded. Thus, any manifold admitting a positive Einstein metric must also have finite fundamental group, immediately excluding many topological types. Meanwhile, Yamabe problem \citep{yamabe1960deformation} considers whether a conformal class contains a metric of constant positive scalar curvature. Examples of Einstein manifolds in these higher dimensions include the Calabi-Yau manifolds, G2-manifolds, and Sasaki-Einstein spaces. A general classification remains out of reach.

\section{The Parallels Between Minimal Surfaces and Einstein Four-Manifolds}

The purpose of this chapter is to explore a number of parallels between minimal surfaces embedded within a three-manifold, and Einstein four-manifolds. These analogies were noted in \citep{SongMainReference}, building upon several prior developments in geometry and analysis. We attempt to articulate each parallel as clearly as possible, providing formal definitions, proofs where necessary, and referencing the original papers, so that the interested reader may investigate each idea in more depth, should they so wish.

Let us note that at no point do we claim that these objects are the same, in fact they are distinct and live in different spaces. However, this distinction makes the parallels even more intriguing.

\subsection{Variational Structures}

We begin with the most natural parallel: both minimal surfaces and Einstein metrics arise as critical points of geometric functionals, defined on spaces of embeddings or metrics, respectively.

Let $ M $ be a surface immersed in a Riemannian manifold $ (N, g) $. The area functional is given by:

\begin{equation}
     \mathcal{A} = \int_{M} \mathrm{d}A,
\end{equation}

where $ \mathrm{d}A $ denotes the induced volume form on $ M  $. A surface is minimal if and only if it is a critical point of $ \mathcal{A} $ under compactly supported variations of the immersion.

On the space of Riemannian metrics, the corresponding functional is the Einstein-Hilbert functional:

\begin{equation}
    \mathcal{E}(g) = \int_{M} R_{g} d \mathrm{vol}_{g},
\end{equation}

where $ R_{g} $ is the scalar curvature associated to the metric $ g $, and $ d \mathrm{Vol}_g $ its volume form.

This variational similarity deepens if we consider the geometric flows naturally associated with each structure.

\subsubsection{Mean Curvature Flow}

Mean Curvature Flow (MCF) is a geometric evolution equation in which a hypersurface deforms over time in the direction of steepest descent for area. The flow tends to smooth irregularities, but singularities can develop where the hypersurface pinches off, reflecting changes in the topology.

\begin{definition}[Mean Curvature Flow]
    Let $ {M_t}, {t \geq 0}$ be a family of hypersurfaces evolving smoothly in a Riemannian manifold. The mean curvature flow is given by:

        \begin{equation}
        \frac{\partial X}{\partial t} = - H \nu,
    \end{equation}
    
where: $ X: M^{n} \times [0, T) \rightarrow \mathbb{R}^{n+1} $ is a smooth family of immersions, $ H $ is the mean curvature, and $ \nu $ is the unit normal vector field \citep{nunemacher2003surface}.

\end{definition}

Minimal surfaces correspond to fixed points of this flow.

\subsubsection{Ricci Flow}

The intrinsic counterpart to MCF is the Ricci flow, introduced by Hamilton in 1982. It deforms a Riemannian metric in the direction of steepest descent for total scalar curvature under appropriate constraints on the volume, smoothing out curvature irregularities.

\begin{definition}[Ricci Flow]
    Let $ M $ be a Riemannian manifold, and $ g(t) $ be a one-parameter family of Riemannian metrics on $ M $. The Ricci flow evolves according to:

    \begin{equation}
        \frac{\partial g}{\partial t} = -2 \operatorname{Ric}_{g},
    \end{equation}
    
    \citep{chow2011ricci, brendle2010ricci}.
    
\end{definition}

Under suitable conditions and normalisation, Einstein metrics arise as fixed points of this flow.

\medskip

Both flows share an underlying philosophy: evolve the geometric structure (extrinsically in the case of mean curvature, and intrinsically in the case of Ricci curvature) toward a state of balance. In both settings, the development of singularities not only signals the breakdown of smooth evolution, but often reveals deep topological features of the underlying space.

\subsection{Second Variations}

The variational narritive developed so far extends to second-order: both minimal surfaces and Einstein metrics possess second variation formulas, whose associated differential operators encode their stability properties under infinitesimal deformations.

\subsubsection{Minimal Surfaces and Second-Variation}

Let $ M \subset N $ be a minimal surface embedded in a Riemannian manifold, and let $ \nu $ denote a smooth normal vector field along $ M $. The second variation of the area functional, under compactly supported normal variation, is given by:

\begin{equation}
    \delta^2 \mathcal{A}(V, V) = \int_M \langle \mathcal{J}V, V \rangle \, \mathrm{d}\mu.
\end{equation}

 Let us briefly define  $ \mathcal{J} $, which is the Jacobi operator.

 \begin{definition}[Jacobi Operator]

    The Jacobi operator is a second-order, elliptic, self-adjoint differential operator acting on normal vector fields. Its general form is:

\begin{equation}
    \mathcal{J} = \Delta^\perp + |A|^2 + \operatorname{Ric}_N(\nu, \nu),
\end{equation}

where $ \Delta $ is the Laplace-Beltrami operator on $ M $, $ A $ is the second fundamental form, and $ \operatorname{Ric}_N$ is the Ricci curvature of the ambient manifold $ N $ in the direction of the unit normal vector field $ \nu $ \citep{fischerschoen1980}.
     
 \end{definition}

A minimal surface is said to be stable if the second variation of area is non-negative for all compactly supported variation. Equivalently, a minimal surface is said to be stable if $ \mathcal{J} $ has non-negative spectrum.

\subsubsection{Einstein Metrics and Second-Variation}

Let $ (M^{4}, g) $ be a compact Einstein four-manifold. A variation of the metric is given by a symmetric $ (0, 2) $-tensor $ h \in S^{2}(T^{*}M) $. To isolate meaningful variations, we restrict attention to transverse-traceless tensors:

\begin{equation}
    \mathrm{TT}_g = \left\{ h \in S^2(T^*M) \mid \operatorname{Tr}_g(h) = 0,\; \delta_g h = 0 \right\},
\end{equation}

which preserve volume and are orthogonal to conformal transformations.

The second variation of the Einstein-Hilbert functional, restricted to $ \text{TT}_{g} $, is governed by the Lichnerowicz Laplacian \citep{galloway1989}:

\begin{equation}
    \Delta_L h = \nabla^* \nabla h + 2 \mathring{R}h,
\end{equation}

where $ \nabla^{*} \nabla $ is the rough Laplacian, and $ \mathring{R} $ is the curvature operator acting on symmetric 2-tensors, defined locally by:

\begin{equation}
    (\mathring{R}h)_{ij} = R_{ikjl} h^{kl}.
\end{equation}

The rough Laplacian will appear again later, warranting a formal definition.

\begin{definition}[Rough Laplacian]

    Let $ \nabla_{X} $ denote the covariant differentiation along the vector field $ X $. For a smooth section $ \phi \in \Gamma(E) $, , where $ E $ is a vector bundle over $ M $, the rough Laplacian is defined as:

    \begin{equation}
    \Delta \phi = - \operatorname{Tr}_{g}(\nabla^{2} \phi) = - \sum_{i=1}^{n} \left( \nabla_{e_i} \nabla_{e_i} \phi - \nabla_{\nabla_{e_i} e_i} \phi \right),
\end{equation}

   where $ \{e_i\} $ is a local orthonormal frame. The operator $ \Delta $ is second-order, elliptic, and self-adjoint.
    
\end{definition}

An Einstein metric is said to be linearly stable if the second variation is non-negative for all such $ h $. This is equivalent to the Lichnerowicz operator being non-negative on the space of TT-tensors.

\medskip

Both the Jacobi and Lichnerowicz operators are elliptic, self-adjoint, and have discrete spectra. The index of the structure is defined by the number of negative eigenvalues of the corresponding operator. It quantifies the number of independent directions of instability.

This spectral framework unifies the notion of stability across both minimal and Einstein geometries, and will serve as a foundation for our discussion or compactness and degeneration phenomena, in sections to come.

\section{Monotonicity and Local Rigidity}

Both minimal surfaces and Einstein manifolds obey natural monotonicity formulae, which impose strong constraints on the behaviour of geometric quantities under scaling. These formulae are not mere technicalities, but instead they encode curvature information and guide our understanding of the passage from local to global geometry in  analysis and convergence theory.

\subsection{Minimal Surfaces}

Let $ \Sigma \subset N $ be a minimal surface immersed in a Riemannian $ 3 $-manifold $ N $. For any point $ p \in \Sigma $ and sufficiently small $ r > 0 $, the function:

\begin{equation}
r \mapsto \frac{\operatorname{Area}(\Sigma \cap B(p, r))}{\pi r^2}
\end{equation}

is monotone non-decreasing up to curvature correction terms when the ambient space $ N $ is not flat. In the case $ N = \mathbb{R}^{3} $, the quantity is strictly monotonic and becomes constant if and only if $ \Sigma $ is a plane in a neighbourhood near $ p $ \citep{meeks2012survey, brendle2023minimal, white2016lecturesminimalsurfacetheory}. 

This reflects the second-order minimality condition satisfied by $ \Sigma $, where the rate of area growth relative to $ r^2 $ near a point quantifies how close the surface is to being flat. Thus, it serves as a regularity tool and as a control on local geometry.

\subsection{Einstein Manifolds}

Let $ (M^{4}, g) $ be an Einstein manifold. For any $ p \in M $ and sufficiently small $ r > 0 $, the normalised volume ratio:

\begin{equation}
r \mapsto \frac{\operatorname{Vol}(B(p, r))}{r^4},
\end{equation}

is almost monotone non-increasing as $ r \rightarrow 0 $, with equality if and only if the metric is locally flat in a neighbourhood of $ p $. This result follows from the Bishop-Gromov volume comparison theorem, which applies under a lower bound on Ricci curvature. Since Einstein manifolds satisfy $\mathrm{Ric}_g = \lambda g$, they fall naturally into the scope of this comparison principle.

In this context, deviation from constancy in the normalised volume function measures how the Einstein condition influences local volume growth. As in the minimal surface case, the monotonicity property provides both a local rigidity result and a powerful analytical tool. It plays a central role in compactness and convergence theorems, particularly within the Cheeger–Gromov theory of collapsing sequences of Riemannian manifolds.

\medskip

The presence of monotonicity formulae in both settings illustrates a deeper structural similarity, with each expressing a local rigidity. Both formulae become equality statements in flat geometry and deviation from monotonicity measures curvature in a quantifiable way. These monotonicity restrictions thus play a dual role: they regulate geometry and hint at underlying analytical principles governing stability, regularity, and convergence in geometric flows.

\section{Epsilon-Regularity and Energy Thresholds}

Epsilon-regularity is a foundational tool in geometric analysis, particularly in the study of geometric partial differential equations. Though this tool represents a concept which is rich in analytic formalism, the principle is simple: under smallness assumptions on an integral curvature quantity, one can deduce strong pointwise control of the geometry. In this way, integral control implies regularity, and singularities may be ruled out below critical energy thresholds. We refer the reader to \citep{Tao_2009b} for an accessible introduction to this principle in broader contexts.

\subsection{Minimal Surfaces}

For minimal surfaces, a classical result due to Choi and Schoen proposes establishes that singularities are excluded in regions where the total curvature is sufficiently small. In particular, embedded minimal surfaces in a Riemannian three-manifold enjoy pointwise curvature bounds, provided the $ L^{2} $-norm of the second fundamental form remain below a universal threshold \citep{haslhofer2024lecturesmeancurvatureflow}.

\begin{theorem}[Choi-Schoen]
    There exist constants $\varepsilon > 0 $ and $ C > 0 $ such that the following holds. Let $ \Sigma $ be an embedded minimal surface in a Riemannian 3-manifold $ N $, and let $  p \in \Sigma $. If $ r > 0 $ is sufficiently small and:

    \begin{equation}
        \int_{\Sigma \cap B(p, r)} |A|^2 < \varepsilon,
    \end{equation}

    then:

    \begin{equation}
        \sup_{\Sigma \cap B(p, r/2)} |A|^2 \leq \frac{C}{r^2},
    \end{equation}

    where $|A|$ denotes the norm of the second fundamental form of $\Sigma$ \citep{ChoiSchoen}.
    
\end{theorem}

This result provides a crucial regularity criterion: if a minimal surface exhibits sufficiently small curvature energy in a region, then the pointwise curvature remains controlled on a smaller concentric ball. This tool underlies many compactness results and blow-up analyses in minimal surface theory.

\begin{proof}[Sketch of Proof]

The proof follows a classical bootstrapping technique, starting from small $ L^{2} $-control on the second fundamental form, and deriving a pointwise curvature bound.

Let \( \Sigma \subset N^3 \) be an embedded minimal surface. The second fundamental form $ A $ satisfies the Simons inequality:

\begin{equation*}
    \tfrac{1}{2} \Delta |A|^2 \geq |\nabla A|^2 - C_1 |A|^4 - C_2 |A|^2,
\end{equation*}

    where $C_{1},C_{2}$ depend only on a bound for $\|{\mathrm{Rm}}_{N}\|_{C^{1}}$ in the chosen coordinate patch

    Let $ \phi \in C^{\infty}_{c}(B_{\Sigma}(p,r)) $ be a smooth cut-off function satisfying:

    \begin{equation*}
          0\le\phi\le1,\quad
  \phi\equiv1\text{ on }B(p,r/2),\quad
  |\nabla\phi|\le \frac{2}{r}.
    \end{equation*}

    Multiply Simons inequality by $ \phi^{2} $ and integrate over $ \Sigma $:

    \begin{align*}
    \int_{\Sigma} \phi^2 \Delta |A|^2 
    &\geq \int_{\Sigma} \left( 2\phi^2 |\nabla A|^2 - 2C_1 \phi^2 |A|^4 - 2C_2 \phi^2 |A|^2 \right), \\
    \text{while integration by parts yields:}
    \int_{\Sigma} \phi^2 \Delta |A|^2 
    &= -\int_{\Sigma} \langle \nabla |A|^2, \nabla (\phi^2) \rangle 
    \leq \int_{\Sigma} 4\phi |\nabla \phi| |A| |\nabla A|.
\end{align*}

    Applying the Cauchy-Schwarz and Young inequalities gives:

    \begin{equation*}
    4\phi |\nabla \phi| |A| |\nabla A| 
    \leq \tfrac{1}{2} \phi^2 |\nabla A|^2 + \tfrac{8}{r^2} |A|^2.
\end{equation*}

    Combining these estimates:

    \begin{equation*}
    \int_{\Sigma} \phi^2 |\nabla A|^2 + C_1 \int_{\Sigma} \phi^2 |A|^4 
    \leq \frac{C}{r^2} \int_{\Sigma \cap B(p, r)} |A|^2.
\end{equation*}

    Note that in particular:

    \begin{equation*}
    \int_{\Sigma \cap B(p, r)} |A|^4 \leq \frac{C}{r^2} \int_{\Sigma \cap B(p, r)} |A|^2.
\end{equation*}

    The final step is to apply a Moser-type iteration argument, which is a bootrastrapping process for elliptic PDEs. By the Michael-Simon Sobolev inequality on minimal surfaces:

    \begin{equation*}
    \left( \int (\phi |A|)^4 \right)^{1/2} \leq C_S \int \left| \nabla(\phi |A|) \right|^2.
\end{equation*}

    With the previous $ L^{4} $-bound, one obtains a recursive estimate:

    \begin{equation*}
    |A| _{L^{2p}(B(p, r/2))} \leq \frac{C_p}{r^{\alpha_p}} |A|_{L^p(B(p, r))}, \quad p = 2, 4, 8, \dots
\end{equation*}

    With Moser-type interation, we send $ p \to \infty $ and deduce the desired bound:

    \begin{equation*}
    \sup_{B(p, r/2)} |A|^2 \leq \frac{C}{r^2} \int_{B(p, r)} |A|^2.
\end{equation*}

    Finally, choosing $ \varepsilon < 1/C $, one obtains:

    \begin{equation*}
    \sup_{B(p, r/2)} |A|^2 < \frac{1}{r^2},
\end{equation*}

    which implies the claimed pointwise estimate:

    \begin{equation*}
    \sup_{B(p, r/2)} |A|^2 \leq C r^{-2}.
\end{equation*}

    \citep{chen1984total, Simon1983, brendle2023minimal}
    
\end{proof}

\subsection{Einstein Manifolds}

An analogous epsilon-regularity result holds for Einstein four-manifolds. In foundational work by Anderson, and later contributions by Cheeger-Tian, Nakajima, and Gao, it was shown that pointwise curvature bounds can be derived from the integral $ L^{n/2}$-control of the Riemann curvature tensor on small balls. In dimension four, the critical exponent $n/2$ for curvature control coincides with $2$.

\begin{theorem}[Anderson, Cheeger–Tian, Nakajima, Gao]

    There exist constants \( \varepsilon > 0 \) and \( C > 0 \) such that the following holds. Let \( (M^4, g) \) be an Einstein manifold, and let \( p \in M \). If \( r > 0 \) is sufficiently small and:

    \begin{equation}
    \int_{B(p, r)} | \mathrm{Rm} |^2 < \varepsilon,
\end{equation}

    then:

    \begin{equation}
    \sup_{B(p, r/2)} | \mathrm{Rm} | \leq \frac{C}{r^2},
\end{equation}

    where $ |Rm| $ denotes the pointwise norm of the Riemann curvature tensor. \citep{Anderson, CheegerTian, nakajima1988compactness, gao1990einstein}
\end{theorem}

This theorem states that if the $ L^{2} $-norm of the total curvature is sufficiently small in a ball $ B(p, r) $, then the curvature remains uniformly bounded on the smaller concentric ball $ B(p, r/2) $, proving control over the geometry of Einstein manifolds under integral curvature bounds.

\begin{proof}[Sketch of Proof]

Following the method used for proving the Choi-Schoen theorem, we  again follow the standard bootstrap argument, moving from integral to pointwise control.

A key step is the use of harmonic coordinates. Anderson \citep{Anderson} showed that if $ \int_{B(p,r)} |\mathrm{Rm}|^2 < \varepsilon $ for sufficiently small $ \varepsilon $, then there exist harmonic coordinates $ \{x_i\} $ on $ B(p, r) $ such that the metric satisfies:
\begin{equation}
    \| g_{ij} - \delta_{ij} \|_{C^{1,\alpha}(B(p, r))} \leq \tfrac{1}{2}, \qquad \Lambda^{-1} \delta_{ij} \leq g_{ij} \leq \Lambda \delta_{ij},
\end{equation}

for a uniform $ \Lambda > 1 $.  In these coordinates, the Laplace–Beltrami operator is uniformly elliptic with
controlled constants.

On an Einstein manifold, the Riemann curvature tensor satisfies an elliptic equation of the form:

\begin{equation}
    \Delta \mathrm{Rm} = \mathrm{Rm} * \mathrm{Rm},
\end{equation}

where $ \Delta $ is the rough Laplacian and $ * $ denotes bilinear contractions of curvature terms \footnote{See \citep{besse1987einstein} for more details.} This identity follows from the second Bianchi identity together with the Einstein condition $ \mathrm{Ric} = \lambda g $.

The equation above constitutes a non-linear elliptic system for $ \mathrm{Rm} $, with quadratic lower-order terms. Since the equation is elliptic and the background metric is well-controlled in harmonic coordinates, classical elliptic regularity theory, such as Moser iteration or interior \( L^p \)-to-\( L^\infty \) estimates, can be applied \citep{bando1989construction, gilbarg2001elliptic}. This gives:

\begin{equation}
    \sup_{B(p, r/2)} | \mathrm{Rm} | \leq \frac{C}{r^2} \left( \int_{B(p, r)} | \mathrm{Rm} |^2 \right)^{1/2}.
\end{equation}

If the integral is smaller than a universal threshold \(\varepsilon\), we may absorb the constant into the right-hand-side to reach the desired estimate.

\end{proof}

\section{Compactness and Convergence}

In both minimal surface theory and Einstein geometry, compactness theorems provide essential structure in the presence of degeneration. Sequences of minimal surfaces or Einstein metrics may develop singularities in regions of concentrated curvature; however, under suitable geometric bounds, these singularities are controlled, and the degeneration is highly structured rather than chaotic.

\subsection{Minimal Surfaces}

Compactness theory for minimal surfaces has been developed in increasingly refined ways over the past two decades. An important result due to Sharp, Chodosh, Ketover, and Maximo shows that under uniform area and Morse index bounds, a sequence of minimal surfaces converges to a smooth limiting surface, with their topological type controlled.

\begin{theorem}[Sharp-Chodosh-Ketover-Maximo]

    Let $ \{ \Sigma_i \} $ be a sequence of closed, embedded minimal surfaces in a Riemannian 3-manifold $ (N, g) $, such that:

    \begin{equation}
    \operatorname{Area}(\Sigma_i) \leq C, \qquad \operatorname{index}(\Sigma_i) \leq C.
\end{equation}

    Then, after passing to a subsequence, $ \Sigma_i \to \Sigma_\infty \subset N  $ smoothly on compact subsets. Furthermore, the genus of $ \Sigma_i  $ is uniformly bounded, and the limit $ \Sigma_{\infty}  $ is a smooth, embedded minimal surface. \citep{Chodosh_2017}
    
\end{theorem}

This theorem shows that uniform control of area and index excludes wild topological degeneration. The result plays a foundational role in variational theory, particularly in contexts such as the existence of minimal surfaces in 3-manifolds with positive Ricci curvature.

\begin{proof}[Sketch of Proof]

    Let us begin by proving that the index bound will imply the existence of at most finitely many unstable regions. By definition, $ \operatorname{index}(\Sigma_{i}) \leq C $ implies that there are at most $ C $ linearly independent directions along which the second variation is negative. Equivalently, one can find at most $ C $ disjoint geodesic balls in $ \Sigma_{i} $ on which the surface is unstable. Outside of these geodesic balls, the surface is unstable \footnote{See \citep{Chodosh_2017}  for the local picture of degenerations of bounded-index hypersurfaces}.

    On any stable minimal surface in a 3-manifold, Schoen's curvature estimates give a pointwise bound:

    \begin{equation}
    |A|^2(p) \leq \frac{C}{\operatorname{dist}(p, \partial B)^2},
\end{equation}

    for intrinsic balls $ B \subset \Sigma_{i} $ disjoint from $ \partial \Sigma_{i} $. In particular, on compact susbets away from the unstable regions, one obtains a uniform curvature bound: 

    \begin{equation}
    \sup_{\Sigma_i \setminus \bigcup B_r(x_{i,k})} |A| \leq C.
\end{equation}

    We cover each $ \Sigma_{i} $ by finitely many intrinsic balls of radius $ r > 0 $, chosen to be sufficiently small so that no ball contains more than one unstable center. On each stable ball, our previous estimate gives:

    \begin{equation}
    \sup |A| \leq C(r).
\end{equation}
    
     Since there are at most $ C $ unstable balls, we conclude a uniform curvature bound outside of their union:

    \begin{equation}
    |A|_{\Sigma_i}(x) \leq C', \quad \forall x \in \Sigma_i \setminus \bigcup_{k=1}^{C} B_r(x_{i,k}),
\end{equation}

    where $ x_{i, k} $ denotes the centres of the unstable regions.

    Let us now see that there is smooth subsequential convergence away from points of issue. The uniform curvature bound on $ \Sigma_i \setminus \bigcup B_r(x_{i,k}) $, together with the area bound, allow us to apply Allard's compactness theorem to extract a subsequence converging in $ C^{\infty} $ on compact subsets of:

    \begin{equation}
    N \setminus \{ \text{at most } C \text{ points} \}.
\end{equation}

    The limit is a smooth minimal lamination. However, the area bound and embeddedness prevent multiplicity and ensure that the convergence is to a single embedded minimal surface away from the concentration points.

    In order to control the topology, we use the Gauss-Bonnet theorem to estimate the Euler characteristic:

    \begin{equation}
    2\pi \chi(\Sigma_i) = \int_{\Sigma_i} K_{\Sigma_i} \, dA = \int_{\Sigma_i} \left( \operatorname{Sec}_N(\tau) - \tfrac{1}{2} |A|^2 \right) \, dA,
\end{equation}

    where $ \operatorname{Sec}_{N}(\tau) $ is the ambient sectional curvature in the direction of the tangent plane. Since $ \operatorname{Sec}_{N} $ is bounded and $ \int |A|^{2} dA $ is controlled by $ \operatorname{Area} \cdot \sup |A|^2 $, this gives a uniform bound on $ \chi(\Sigma_{i}) $, and hence on the genus.
    
\end{proof}

\subsection{Einstein Four-Manifolds}

A similar philosophy holds for sequences of Einstein four-manifolds. Building on work by Anderson, Bando, Nakajima, Gao, Cheeger, and Tian, it has been shown that a sequence of Einstein manifolds with uniform topological and geometric bounds converges, modulo singularities, to a smooth limiting Einstein orbifold.

\begin{theorem}[Anderson-Bando-Nakajima-Gao]

    Let $ \{ (M_{i}, g_{i}) \} $ be a sequence of Einstein four-manifolds satisfying:

    \begin{equation}
    \chi(M_i) \leq C, \qquad \operatorname{Vol}(M_i, g_i) \geq C^{-1}, \qquad \operatorname{diam}(M_i, g_i) \leq C.
\end{equation}

    Then after passing to a subsequence, $ (M_i, g_i) \to (X, g_\infty) $ in the Gromov–Hausdorff sense, where $ X $ is a smooth Einstein orbifold. Furthermore, the number of diffeomorphism types of $ M_{i} $ is finite. \citep{Anderson, bando1989construction, gao1990einstein}
    
\end{theorem}

As in the minimal surface case, degeneration is controlled. Singularities arise only at isolated points, and the limiting space retains a high degree of regularity elsewhere. These results are central in the study of moduli spaces of Einstein metrics.

\begin{proof}

    Fix a small $ \varepsilon $ and a volume threshold $ v_{0} > 0 $. For each $ p \in M_{i} $, define the curvature scale:

    \begin{equation}
    r_\varepsilon(p) = \sup \left\{ r \leq 1 \,\middle|\, \int_{B(p, r)} | \mathrm{Rm} |^2 \leq \varepsilon \right\}.
\end{equation}

    Decompose the manifold in the following way:

    \begin{equation}
    M_i^{>} = \left\{ p \in M_i \,\middle|\, \operatorname{Vol}(B(p, r_\varepsilon(p))) > v_0 r_\varepsilon(p)^4 \right\}, \qquad
    M_i^{\leq} = M_i \setminus M_i^{>}.
\end{equation}

    Here, $ M_{i}^{>} $ represents the thick region, where volume at scale $r_{\varepsilon}$ is non–collapsing, and $M_i^{\le}$ is the thin region (note, thick/thin decomposition will be officialy introduced in the next section, as it represents its own parallel).

    By $ \varepsilon $-regularity, every $ p \in M_{i}^{>} $ lies in a ball on which $ | \operatorname{Rm} | $ is uniformly bounded. Coupling this with the global diameter bound $ \operatorname{diam}(M_{i}) \leq C $, and the global non-collapsing $ \operatorname{Vol}(B(p, 1)) \geq C^{-1} $, this implies higher-order control via usual elliptic regularity.

    By Cheeger-Gromov-Hamilton compactness,  a subsequence converges smoothly (with multiplicity one) on $ M_{i}^{>} $:

    \begin{equation}
    (M_i^{>}, g_i) \longrightarrow (X_{\mathrm{reg}}, g_\infty),
\end{equation}

    where \( X_{\mathrm{reg}} \) is an open dense subset of the limiting Einstein orbifold.

    On the thin region $ M_{i}^{\leq} $, the geometry collapses along small-volume balls. The volume collapse at scale $ \varepsilon $ forces the pointed rescalings $ \bigl(B(p,r_{\varepsilon}(p)),r_{\varepsilon}(p)^{-2}g_i\bigr) $ to converge to a complete Ricci-flat ALE space $ \mathbb{R}^{4} / \Gamma $. Each of these limits contributes an isolated orbifold singularity of type $\mathbb{R}^4/\Gamma$.

    Gluing the smooth convergence on $ M_{i}^{>} $ with the orbifold asymptotics of $ M_{i}^{\leq} $ yields:

    \begin{equation*}
         (M_i, g_i) \longrightarrow (X, g_\infty),
    \end{equation*}

    where $X$ is a smooth Einstein orbifold with finitely many singular points.

    The final step is an appeal to an orbifold finiteness theorem (originally due to Anderson-Cheeger, and then refined by Bando-Kasue-Nakajima): \textit{ Given uniform bounds on $ \chi(M_{i}) $, $ \operatorname{Vol}(M_{i}) $, and $ \operatorname{diam}(M_{i}) $, there are only finitely many homeomorphism types in the smooth pre-limit sequence.} That is, the possible arrangements of finitely many orbifold singular points and the discrete data of their local groups $ \Gamma \subset O(4) $ admits only finitely  many combinations under these constraints.
    
\end{proof}

\section{Thick/Thin and Sheeted/Non-Sheeted Decomposition}

In geometric analysis, it is often fruitful to decompose a complicated space into regions of controlled and uncontrolled geometry. This dichotomy appears both in the theory of Einstein four-manifolds and in the study of minimal surfaces, whereby one seperates the domain into a thick (or sheeted) part with regular geometric behaviour, and a thin (or non-sheeted) part where curvature concentrates and singularities arise. These decompositions, though originating in distinct contexts, serve a parallel purpose of isolating regions amenable to compactness theorems from those requiring finer analysis.

\subsection{Thick/Thin Decomposition}
 
In the setting of Einstein four-manifolds, the thick/thin decomposition is based on local pointwise control of curvature energy and volume ratios. Fix a small constant $ \varepsilon > 0 $ (determined by epsilon-regularity results) to serve as a threshold for acceptable curvature concentration. At each $ p \in M $, define the regularity scale $ r_{\varepsilon(p)} $ as:

\begin{equation}
    r_\varepsilon(p) := \sup \left\{ r \in (0, 1] \;\middle|\; \int_{B(p, r)} |\mathrm{Rm}|^2 \leq \varepsilon \right\}
\end{equation}

This radius captures the largest scale at which curvature energy remains below the critical threshold. To distinguish collapsing versus non-collapsing behabiour, we examine the volume of the corresponding ball:

\begin{itemize}
    \item If $ \operatorname{Vol}(B(p, r_\varepsilon(p))) > V_0 \cdot r_\varepsilon(p)^4 $, then $ p $ lies in the thick region.
    \item Otherwise, $ p $ lies in the thin region, where volume collapse may occus and curvature can become unbounded.
\end{itemize}

This classification allows one to localise analysis: the thick region supports regularity theorems and compactness results, while the thin region demands delicate rescaling techniques and bubbling analysis.

This process is formalised via the following:

\begin{definition}[Thick/thin decomposition for Einstein 4-manifolds]
    Let $ (M^{4}, g) $ be an Einstein manifold. For fixed constants $ \varepsilon > 0 $ and $ V_0 > 0 $, define:

    \begin{align}
    r_\varepsilon(p) &:= \sup \left\{ r \in (0,1] : \int_{B(p, r)} |Rm|^2 \leq \varepsilon \right\}, \\
    M_{> V_0} &:= \left\{ x \in M : \operatorname{Vol}(B(x, r_\varepsilon(x))) > V_0 \cdot r_\varepsilon(x)^4 \right\}, \\
    M_{\leq V_0} &:= \left\{ x \in M : \operatorname{Vol}(B(x, r_\varepsilon(x))) \leq V_0 \cdot r_\varepsilon(x)^4 \right\}.
\end{align}
\end{definition}

\subsection{Sheeted/Non-Sheeted Decomposition}

For minimal surfaces, an analogous decomposition exists. Rather than focusing on volume growth, one tracks area concentration within stable regions. The goal is to distinguish between points where the surface behaves like multiple well-separated sheets versus regions where geometric degeneration may occur. 

Following the framework introduced by Song \citep{song2022morseindexbettinumbers}, fix a small radius $ r > 0 $ and a threshold $ n_{0} > 0 $. At each point $ p \in \Sigma $, define the stability radius $ s(p) \leq r $ as the largest radius such that the portion $ \Sigma \cap B(p, s(p)) $ is stable. Then we classify points as follows:

\begin{itemize}
    \item If \( \operatorname{Area}(\Sigma \cap B(p, s(p))) > n_0 \cdot s(p)^2 \), we say \( p \) lies in the non-sheeted region, indicating possible bubbling or curvature concentration.
    \item If \( \operatorname{Area}(\Sigma \cap B(p, s(p))) \leq n_0 \cdot s(p)^2 \), then \( p \) lies in the sheeted region, where the surface behaves regularly.
\end{itemize}

Intuitively, the sheeted region is like a surface composed of well-separated layers, or `sheets,' with predictable behaviour. In contrast, the non-sheeted region may exhibit folding, bunching, or bubbling, signaling geometric degeneration.

\begin{definition}[Sheeted/Non-Sheeted Decomposition for Minimal Surfaces]
Let \( \Sigma \subset (N^3, g) \) be a minimal surface. Fix constants \( r > 0 \) and \( n_0 > 0 \). Define:

\begin{align}
    s(p) &:= \sup \left\{ \tilde{r} \leq r : \Sigma \cap B(p, \tilde{r}) \text{ is stable} \right\}, \\
    \Sigma_{> n_0} &:= \left\{ x \in \Sigma : \operatorname{Area}(\Sigma \cap B(x, s(x))) > n_0 \cdot s(x)^2 \right\}, \\
    \Sigma_{\leq n_0} &:= \left\{ x \in \Sigma : \operatorname{Area}(\Sigma \cap B(x, s(x))) \leq n_0 \cdot s(x)^2 \right\}.
\end{align}
\end{definition}

\subsection{Conclusion}

Though these two decompositions arise in very different settings, they share a common philosophical role. Both seek to isolate regions of geometric regularity from those of potential degeneration. In the thick or sheeted regions, the geometry behaves predictably, and powerful compactness theorems apply. In the thin or non-sheeted regions, by contrast, curvature concentrates, topological change can occur, and delicate analysis is required.

\section{Reflections on the Parallels and Directions for Further Study}

The parallels proposed span the fields of topology, differential geometry, and geometric analysis. Their recurrence across seemingly disparate contexts suggests a deeper conceptual link between minimal surfaces and Einstein four-manifolds. Pursuing this intuition could yield fresh insights at the intersection of geometric analysis and global differential geometry.

Thus, we outline several potential directions for further study:

\begin{itemize}

    \item Characterising the class of Einstein four-manifolds that admit isometric minimal embeddings into higher-dimensional ambient spaces. Such a classification may generalise contructions arising from Hitchin's and LeBrun's work on self-dual manifolds and twistor theory \citep{Hitchin1981, LeBrun1991twistor, LeBrun1994Einstein}, and may provide some form of bridge between intrinsic and extrinsic geometry.
    
    \item Analyse whether these structural similarities carry over to other elliptical system, such as Yang–Mills fields or metrics of constant scalar curvature. Many of these systems may be viewed through the lens of moment maps in infinite-dimensional gauge-theoretic settings \citep{DonaldsonKronheimer, Salamon_1989}, suggesting a unifying analytical framework for studying regularity, moduli spaces, and bubbling phenomena.
    
    \item Classical minimal surface theory possesses a deep holomorphic structure via the Weierstrass–Enneper representation. On the Einstein side, complexified tools, such as the study of complex holonomy, self-duality, or twistor spaces, often reveal hidden integrable structures. Exploring a possible geometric bridge between these complexifications might uncover a shared algebraic or symplectic underpinning.

    \item Plateau and free-boundary problems for minimal surfaces can be compared with conformal, Robin-type, or Bartnik boundary conditions in the context of Einstein metrics \citep{AndersonBoundary, Bartnik1989}. Investigating conditions under which these boundary theories align could inform existence and regularity results, particularly in the presence of scalar curvature constraints.

\end{itemize}

Each of these questions present a compelling avenue for deeper mathematical investigation. In the remainder of this work, we choose to focus on the first question: establishing criteria under which an Einstein four-manifold may be realised as a minimal submanifold in a suitable higher-dimensional ambient space, constructing an explicit example to demonstrate this. 

\section{Embedding Einstein Four-Manifolds as Minimal Submanifolds}

In this section, we explore a concrete connection between Einstein geometry and the theory of minimal submanifolds. Specifically, we consider conditions under which an Einstein four-manifold may be isometrically immersed as a minimal submanifold in a higher-dimensional ambient space. This framework allows us to study Einstein manifolds using techniques from minimal surface theory, and provides a geometric realisation of the parallels explored throughout this exposition.

To prepare the ground, we begin by examining a classification result due to Derdziński, which is a stronger result, though reminscent of Jensen's theorem \citep{JensonTheorem}. This remarkable theorem shows that, under mild geometric assumptions, the Einstein condition forces a strong symmetry. This reduction paves the way for our embedding results by narrowing the class of manifolds to those whose structure is particularly amenable to minimal immersion techniques.

\subsection{Jensen–Derdziński Classification}\label{sec: Jensen-Derdzinski}

The defining feature of an Einstein manifold is constant Ricci curvature, which is a strong constraint to place on a group of objects. In 1969, Jensen discovered that all locally homogeneous Einstein four-manifolds are in fact locally symmetric. More recently, Derdzinski generalised this result, showing that if an Einstein four-manifold is self-dual, locally irreducible, and has non-zero scalar curvature, then it must also be locally symmetric (save for a short explicit list of exceptional Kahler-Einstein spaces). Thus the Einstein equation, when coupled with mild symmetry or chirality hypotheses, forces the manifold into an even narrower corner of the geometric menagerie.  

We will start by presenting Jensen's theorem as a warm-up, and then pave the way to the more general result.

\begin{definition}[Locally Symmetric Space]

    A Riemannian manifold $ (M, g) $ is locally symmetric if its Riemann curvature tensor is parallel:

    \begin{equation}
        \nabla R = 0
    \end{equation}
    
\end{definition}

\begin{theorem}[Jensen's Theorem]\label{thm:Jensen}
    Every locally homogeneous Riemannian Einstein four-manifold is locally symmetric \citep{JensonTheorem}. 
\end{theorem}

Jensen's original proof proceeded by a brute-force classification of four-dimensional Lie algebras. We will not present the proof here, but the keen reader may consult \citep{JensonTheorem} for the details.

This result has since been extended by Derdziński, who removed the homogeneity assumption while retaining the rigidity of symmetry under weaker geometric hypotheses.

\begin{theorem}[Derdziński Rigidity]
Let $ (M^{4}, g) $ be a compact, oriented Einstein four-manifold. Suppose $ g $ is self-dual, locally irreducible, and has nonzero scalar curvature. Then $ (M, g) $ is locally symmetric, except in the following two Kähler–Einstein cases:

\begin{equation}
    \left( \mathbb{CP}^{2} \# \overline{\mathbb{CP}}^{2},\, g_{\mathrm{Page}} \right), \qquad 
\left( \mathbb{CP}^{2} \# 2\, \overline{\mathbb{CP}}^{2},\, g_{\mathrm{CLW}} \right).
\end{equation}

\end{theorem}

The proof  will use the idea of the Weyl tensor consistently. Thus, a formal definition is most certainly warranted. 

\begin{definition}[Weyl tensor]
    Let $ (M^{n}, g), n \geq 3 $ be a Riemannian manifold, with:

    \begin{equation}
          R_{ijkl},\quad
  R_{ij}=R^{k}{}_{ikj},\quad
  R=g^{ij}R_{ij}.
    \end{equation}
    
    The Weyl tensor is the trace-free part of $ R $:

    \begin{equation}\label{eq:WeylTensor}
\begin{aligned}
  W_{ijkl}
  &= R_{ijkl}\\
  &\quad-\frac{1}{n-2}\Bigl(
       g_{ik}R_{jl}-g_{il}R_{jk}+g_{jl}R_{ik}-g_{jk}R_{il}
     \Bigr)
  +\frac{R}{(n-1)(n-2)}
     \Bigl(
       g_{ik}g_{jl}-g_{il}g_{jk}
     \Bigr).
\end{aligned}
\end{equation}

\end{definition}

The Weyl tensor shares all algebraic symmetries of $ R_{ijkl} $, is totally trace-free, and vanishes precisely when $ (M^{n}, g) $ is locally conformally flat.

\begin{proof}[Sketch of Proof]

    Let $ (M^4, g) $ be an Einstein manifold satisfying the hypotheses of the 
    theorem: namely, $ g $ is self-dual, locally irreducible, and has non-zero scalar curvature $ s $. 

    \vspace{0.5em}

    \noindent\textbf{Decomposition of Curvature.}
    
    Since $\operatorname{Ric} = \frac{s}{4}g $, the Riemann curvature tensor decomposes as:

    \begin{equation}
         R = W^{+} + \frac{s}{12}\, g \owedge g,
    \end{equation}

    where $ g \owedge g $ denotes the Kulkarni–Nomizu product, and $ W^{+}\colon \Lambda^{+} \to \Lambda^{+} $ is the self-dual Weyl curvature operator.

    \vspace{0.5em}
\noindent\textbf{Diagonalisation and Spectral Constancy of $W^{+}$.}

    The self-dual Weyl tensor $ W^{+}\colon\Lambda^{+}\!\to\!\Lambda^{+} $ is self-adjoint and trace-free, and hence is pointwise diagonalisable with real eigenvalues $ \lambda_{1}, \lambda_{2}, \lambda_{3} $, satisfying $ \lambda_{1} + \lambda_{2} + \lambda_{3} = 0 $. 
    
    By Remark 10.3 of Derdziński's paper, the assumption that $ s \neq 0 $ ensures that the eigenvalues and their corresponding eigenforms  can be chosen smoothly across $ M $. That is, there exists a local orthonormal frame $ \{ \alpha_{1}, \alpha_{2}, \alpha_{3} \} $ of $ \Lambda^{+} $ such that:

    \begin{equation}
           W^{+}(\alpha_j) = \lambda_j \alpha_j, \qquad j = 1,2,3,
    \end{equation}

    and each $ \lambda_j \in \mathbb{R} $ is constant on $ M $.

    \vspace{0.5em}
\noindent\textbf{Case-by-Case Analysis.}

    We distinguish two cases, according to the spectrum of $ W^{+} $.
    
    \begin{enumerate}
        \item All eigenvalues are pairwise distinct.
        \item $ W^{+} $ has a repeated eigenvalue.
    \end{enumerate}

    Let us start by considering the first case. If all eigenvalues are pairwise distinct, the stabiliser of $ W^{+} $ under the adjoint action of $ \operatorname{SO}(3) $ is trivial. We now use the Singer-Thorpe proposition:

    \begin{proposition}[Singer-Thorpe]
        If the curvature operator of a four-dimensional Riemannian Einstein manifold is self-adjoint, trace-free, and has constant eigenvalues, then the curvature is parallel, $ \nabla R = 0 $ \citep{SingerThorpe1970}.
    \end{proposition}
    
    Applying to our scenerio, we have that $ \nabla W^{+} = 0 $, and since $ g $ is Einstein, $ \nabla R = 0 $. 
    
    Further, we can apply the Ambrose-Singer theorem:

    \begin{theorem}[Ambrose-Singer]
        If the Riemann curvature tensor is parallel, then the manifold is locally symmetric.
    \end{theorem}
    
    Then this implies that $ (M, g) $ is locally symmetric in this case.

    Let us now consider what happens if a double eigenvalue occurs. Suppose for instance that $ \lambda_{1} = \lambda_{2} \neq \lambda_{3} $. Then the corresponding two-dimensional eigenspace $ E = E = \operatorname{span}\{\alpha_1, \alpha_2\} \subset \Lambda^{+} $ defines an almost-complex structure $ J $ on $ M $, which is compatible with $ g $. The self-duality condition forces the associated Kähler form $ \omega(X, Y) = g(JX, Y) $ to be closed, and hence $ J $ is integrable. Thus, $ (M, g, J) $ is Kähler–Einstein.
    
    The final step is to then classify the compact Kahler-Einstein self-dual possibilities. By results from LeBrun and Singer (building on the previous work of Hitchin and Page), the only compact, self-dual, positive-scalar-curvature Kahler-Einstein four-manifolds are: 

    \begin{equation}
           \bigl(\mathbb{CP}^{2}\#\overline{\mathbb{CP}}^{2},\,g_{\mathrm{Page}}\bigr),
   \quad
   \bigl(\mathbb{CP}^{2}\#2\,\overline{\mathbb{CP}}^{2},\,g_{\mathrm{CLW}}\bigr).
    \end{equation}

    Hence these two metrics exhaust the non-symmetric possibilities.

    We can thus conclude that either:

    \begin{itemize}
\item[(i)] $ W^{+} $ has three distinct eigenvalues, so $ \nabla R = 0 $ and $ (M, g) $ is locally symmetric; or
\item[(ii)] $ W^{+} $ has a repeated eigenvalue, and the manifold lies in the Kähler–Einstein branch, where compactness forces $ (M, g) $ to be one of the Page or Chen–LeBrun–Weber metrics.
\end{itemize}
    
The theorem follows.

\end{proof}

Derdziński’s theorem reveals a profound geometric principle: in dimension four, symmetry is not merely a condition to impose, but can emerge naturally from curvature and topological constraints. Under mild assumptions such as self-duality and irreducibility, the Einstein condition yields local symmetry, collapsing the potential complexity of the manifold’s structure. This rigidity paves the way for our next step: linking these symmetric Einstein manifolds to minimal submanifolds via isometric immersion.

\subsection{Minimal Embeddings of Symmetric Spaces}

The connection between Einstein four-manifolds and minimal submanifolds is made concrete through a series of foundational results in differential geometry, most notably the work of Takahashi on symmetric space immersions.

\begin{theorem}[Takahashi]\label{thm:Takahashi}

Every irreducible compact symmetric space admits a minimal isometric immersion into a Euclidean  sphere \citep{takahashi1966minimal}.
    
\end{theorem}

\begin{proof}[Sketch of Proof]
Takahashi first proved that an isometric an isometric immersion $x: M \rightarrow \mathbb{R}^{m+k}$ of a Riemannian manifold $M$ into Euclidean space satisfying $\Delta x = \lambda x$ with $\lambda \neq 0 $, is minimal in a sphere of radius $ r = \sqrt{m / \lambda}$.

Now, let $ M $ be a compact homogeneous Riemannian manifold whose isotropy representation on $ T_{p}M $ is irreducible. For a non-zero eigenvalue $\lambda $, the eigenspace $ V_\lambda=\{f\in C^\infty(M)\mid\Delta f=\lambda f\} $ is finite-dimensional and $ G = \mathrm{Isom}(M) $-invariant.

Choose an orthonormal basis $\{f_1, \dots, f_n\}$ for $V_\lambda$, with respect to this inner product. We obtain a mapping $\tilde{x} : M \rightarrow \mathbb{R}^n$ by

\begin{equation}
    \tilde{x}(p) = (f_1(p), \dots, f_n(p)) \quad \text{for } p \in M.
\end{equation}

The metric pulled back via \(\tilde x\) is then:

\begin{equation}
    \tilde g=\sum_i df_i\otimes df_i.
\end{equation}

This pullback map is $ G $-invariant, and so Schur’s lemma forces $\tilde{g} = c^2 g$ with $c \neq 0$.

We can then rescale by $ c^{-1} $, giving the map $x(p) = \frac{1}{c} \tilde{x}(p)$, which defines an isometric immersion of $M$ into $\mathbb{R}^n$, satisfying $\Delta x = \lambda x$ with $\lambda \neq 0$. By Takahashi’s prior theorem, this immersion is minimal into a sphere.

We conclude that every compact homogeneous Riemannian manifold with irreducible isotropy group admits a minimal isometric immersion into a Euclidean sphere. In particular, every irreducible compact symmetric space admits a minimal isometric immersion as well.
    
\end{proof}

This result shows that highly symmetric spaces can be realised as minimal submanifolds of ambient Euclidean spaces, which offers a geometric lens from which one can view their structure.

Combining these insights, with Jensen's classification, we arrive at the following:

\begin{theorem}[Minimal realisation of locally homogeneous Einstein four-manifolds]

    Let $ (M^{4}, g) $ be a locally irreducible Einstein four-manifold. Then:

    \begin{itemize}
        \item $ M $ is locally symmetric (by Derdzinski’s theorem)
        \item If $ M $ is compact, then it admits a minimal immersion into a Euclidean sphere (by Takahashi's Theorem).
    \end{itemize}
    
\end{theorem}

\begin{proof}
    The first statement follows directly from Derdzinski's theorem which is built on Jensen's original theorem, and the second is a corollary of Takahashi’s result for compact symmetric spaces. 
\end{proof}

This synthesis of symmetry and minimality reveals a surprising unity between seemingly disparate domains. Through the lens of Takahashi’s theorem and the vision of Derdzinski building on the prior work of Jensen, Einstein four-manifolds emerge not merely as solutions to variational problems, but as participants in a broader geometric narrative. One where curvature, symmetry, and immersion speak in harmony.

\subsection{A Concrete Minimal Immersion of an Einstein Manifold}

A natural way to illustrate the ideas of this chapter is to present a specific and striking example: the complex projective plane \(\mathbb{CP}^{2}\), equipped with its normalised Fubini–Study metric, admits a minimal isometric immersion into the round sphere
\(S^{7}\subset\mathbb R^{8}\). This example elegantly synthesises ideas from Riemannian, complex, and algebraic geometry.

\subsubsection{The Geometry of $ \mathbb{CP}^{2} $}

\begin{definition}[Complex Projective Plane]

    The complex projective plane $ \mathbb{CP}^{2} $ is the space of complex lines through the origin in \(\mathbb C^{3}\):

\begin{equation}
    \mathbb{CP}^{2}\;=\;
\bigl(\mathbb C^{3}\setminus\{0\}\bigr)\big/\!\bigl(z\sim\lambda z,\;\lambda\in\mathbb C^{\!*}\bigr).
\end{equation}

\end{definition}

This defines a smooth, compact, complex manifold of complex dimension $ 2 $ (or real dimension $ 4 $). The distinguished Riemannian metric is the Fubini-Study metric,  $ g_{ \mathrm{FS} } $, which has the advantage of being Kähler and of constant holomorphic sectional curvature.

\begin{definition}[Fubini-Study Metric]\label{Fubini-Study-Metric}
In the affine chart \([1:z_{1}:z_{2}]\), set:

\begin{equation}
    g_{i\bar j}\;=\;
\partial_{z_{i}}\partial_{\bar z_{j}}
\log\!\bigl(1+|z_{1}|^{2}+|z_{2}|^{2}\bigr),
\qquad i,j=1,2.
\end{equation}

This Hermitian metric is Kähler with constant holomorphic sectional curvature \(+4\).\citep{matsumoto2018takeuchi}.
    
\end{definition}

Note, we have introduced the notion of the Hermitian metric, so we include the following definition for completeness:

\begin{definition}[Hermitian Metric]

    A Hermitian metric on a complex manifold $ M $ is a Riemannian metric $ g $ which satisfies:

    \begin{equation}
        g(JX, JY) = g(X, Y), \quad \forall X, Y \in T_xM, \ \forall x \in M,
    \end{equation}

    where $ J $ denotes the almost complex structure on $ M $.
    
\end{definition}

\subsubsection{The Veronese Embedding}

In order to embed $ \mathbb{CP}^{2} $ as a submanifold of a sphere, the Veronese embedding is introduced, which is a classical method to realise projective spaces as algebraic subvarieties of higher-dimensional projective space.

\begin{definition}[Veronese Embedding]

    The $ d $-uple Veronese embedding is the morphism:

    \begin{equation}
        \nu_{d}:\mathbb{P}^{n}\longrightarrow\mathbb{P}^{N},
\qquad
[x_{0}\!:\!\dots\!:x_{n}]
\mapsto\bigl[\text{all monomials of degree }d\bigr].
    \end{equation}

    This map realises $ \mathbb{P}^{n} $ as a smooth projective variety of degree $ d^{n} $ in $ \mathbb{P}^{N} $, where $ N = (n+d) -1 $ \citep{harris1992algebraic, shafarevich1994basic}.
    
\end{definition}

The most fundamental example is taking $ n=2,d=2$ , which gives the classical Veronese surface:

\begin{equation}
    \nu_{2}:\mathbb{CP}^{2}\longrightarrow\mathbb{CP}^{5},
\quad
[x:y:z]\longmapsto[x^{2}:xy:xz:y^{2}:yz:z^{2}].
\end{equation}

We would like to explore this in more detail.

\subsubsection{Lifting to the Sphere and Minimality}

Consider the lift of $ \nu_{2} $ to unit-length homogeneous coordinates in $ \mathbb{C}^{6}$:

\begin{equation}
    \tilde\nu_{2}([x:y:z])
=\frac{1}{\|v\|}
\bigl(x^{2},\,xy,\,xz,\,y^{2},\,yz,\,z^{2}\bigr)
\in\mathbb{C}^{6},
\qquad
\|v\|^{2}=|x|^{4}+|x|^{2}|y|^{2}+\dots+|z|^{4}.
\end{equation}

This defines a map:

\begin{equation}
    \tilde{\nu}_2 : \mathbb{CP}^2 \rightarrow \mathbb{S}^{11} \subset \mathbb{C}^6,
\end{equation}

which is a horizontal lift for the Hopf fibration:

\begin{equation}
    \nu_2 = \pi \circ \tilde{\nu}_2.
\end{equation}

Here, \( \pi: \mathbb{S}^{11} \to \mathbb{CP}^{5} \) denotes the generalised Hopf fibration, with \(\nu_{2}=\pi\circ\tilde\nu_{2}\).

\begin{definition}[Hopf Fibration]
The Hopf fibration is the smooth surjective map

\begin{equation}
    \pi: \mathbb{S}^{2n+1} \subset \mathbb{C}^{n+1} \to \mathbb{CP}^n, \quad z \mapsto [z],
\end{equation}

which sends each unit vector $ z \in \mathbb{S}^{2n+1} $ to its complex line through the origin. The fibers are circles $ \mathbb{S}^1 $, and this realises $ \mathbb{S}^{2n+1} $ as a principal $ \mathbb{S}^1 $-bundle over $ \mathbb{CP}^n $ \citep{frankel2011geometry, kobayashi1969foundations}.
\end{definition}

This structure underlies the Veronese embedding’s lift: by normalising homogeneous coordinates, one maps $ \mathbb{CP}^2 $ into $ \mathbb{S}^{11} \subset \mathbb{C}^6 $, and the Hopf fibration ensures that this lift descends correctly back to projective space.

\subsubsection{Minimal Immersion into $ S^{7} $}

Because the image of $ \tilde\nu_{2} $ is orthogonal to the circle fibres, standard submersion theory shows that
$ \tilde\nu_{2} $ is minimal in $ S^{11} $ if and only if $ \nu_{2}$ is minimal in $ \mathbb{CP}^{5} $.

Finally, identify $ \mathbb{C}^{6}\cong\mathbb{R}^{12}$ and observe that the real and imaginary parts of the six complex coordinates span an 8-dimensional real linear subspace $ \mathbb{R}^{8}\subset\mathbb{R}^{12} $.
The image of $ \tilde\nu_{2} $ lies entirely inside the intersection $ S^{11}\cap\mathbb{R}^{8}=S^{7} $, and the induced metric equals $ g_{\mathrm{FS}} $.
Hence we arrive at:

\begin{proposition}
The complex projective plane $ \left( \mathbb{CP}^2, g_{\mathrm{FS}} \right) $ admits a minimal isometric immersion into the round sphere $ \mathbb{S}^7 \subset \mathbb{R}^8 $.
\end{proposition}

Carmo–Wallach showed that this immersion is characterised by the first
non-trivial eigenspace of the Laplacian on $ (\mathbb{CP}^{2},g_{\mathrm{FS}}) $ \citep{doCarmoWallach1971}.

This example concretely demonstrates the key theme of this exposition: that certain Einstein manifolds, particularly those with high symmetry, can indeed be realised as minimal submanifolds of ambient spaces. It gives not only a conceptual link between curvature variational theories, but also a precise geometric construction bridging the two domains.

\section{Conclusion}

This exposition set out to explore the deep geometric parallels between minimal surfaces and Einstein four-manifolds. Though these structures emerge from different dimensions and inhabit distinct corners of differential geometry, they share surprising common ground.

Through a shared variational framework, second variation theory, stability conditions, monotonicity formulae, and compactness theorems, both minimal surfaces and Einstein manifolds reveal a unified analytical and topological structure. Each admits natural decompositions, sheeted versus non-sheeted, or thick versus thin, that separate regions of regularity from zones of degeneration, allowing powerful geometric and analytic tools to take hold.

These parallels are not merely formal. Under appropriate symmetry assumptions, an Einstein manifold may be isometrically immersed as a minimal submanifold into a higher-dimensional ambient space. The explicit example of $ \mathbb{CP}^2 $, minimally immersed into the round sphere $ S^7 $, demonstrates that in special cases, the two theories do not just echo one another, they coalesce.

More broadly, this study offers a perspective on how concepts in geometric analysis resonate across seemingly disparate settings. The unity we observe hints at deeper principles still to be uncovered.

Naturally, many questions remain open. Can similar analogies be drawn in other elliptic systems, such as Yang–Mills theory or constant scalar curvature metrics? Might further investigation into symmetric spaces yield new classification results for Einstein four-manifolds? Could there be a broader framework unifying variational problems through moment map structures or stability conditions?

In closing, this has been both a journey of geometry and a meditation on mathematical philosophy. It shows that even within the abstract world of high-dimensional curvature, there are moments when distinct objects whisper the same truths. That Einstein four-manifolds and minimal surfaces, born of different dimensions, yet bound by shared structure, might speak the same language is a gentle reminder: in mathematics, hidden harmonies are always waiting to be heard.

\bibliographystyle{elsarticle-harv}
\bibliography{References}

\end{document}